\documentclass{amsart}

\usepackage{amssymb,amsthm}
\usepackage{eucal}
\usepackage[shortlabels]{enumitem}
\usepackage{thmtools,thm-restate}
\usepackage{hyperref}
\usepackage{esvect}
\setlist[description]{font=\normalfont}

\usepackage{interval}
\intervalconfig{soft  open  fences}
\usepackage{mathtools}

\declaretheorem[numberwithin=section]{theorem}
\declaretheorem[sibling=theorem]{proposition}
\declaretheorem[sibling=theorem]{lemma}

\declaretheorem[style=definition,sibling=theorem]{definition}

\declaretheorem[sibling=theorem]{fact}

\declaretheorem{conjecture}

\DeclareMathOperator{\cf}{cf}
\DeclareMathOperator{\cof}{cof} 
\DeclareMathOperator{\ot}{ot}
\DeclareMathOperator{\ro}{ro}
\DeclareMathOperator{\dom}{dom}

\newcommand{\seq}[2]{\langle #1 \mid #2 \rangle}

\newcommand{\Col}{\textup{Col}}
\newcommand{\Add}{\textup{Add}}
\newcommand{\ON}{\textup{ON}}

\newcommand{\GCH}{\textup{\textsf{GCH}}}

\newcommand{\ZFC}{\textup{\textsf{ZFC}}}

\newcommand{\MM}{\textup{\textsf{MM}}}

\newcommand{\PFA}{\textup{\textsf{PFA}}}

\newcommand{\rest}{\upharpoonright}

\renewcommand{\P}{\mathbb P}

\newcommand{\F}{\mathcal F}

\newcommand{\B}{\mathbb B}

\newcommand{\T}{\mathbb T}
\renewcommand{\S}{\mathbb S}
\newcommand{\G}{\mathcal G}

\newcommand{\gcut}{\textup{\textsf{cut}}}
\newcommand{\gchoose}{\textup{\textsf{choose}}}

\newcommand{\bq}{\textup{``}}
\newcommand{\eq}{\textup{''}}

\newcommand{\candc}{\textsf{\textup{c\&c}}}

\DeclareMathOperator{\prt}{part}

\author{Maxwell Levine}

\title{On Determinacy for Cut and Choose Games of Uncountable Length}

\begin{document}

\begin{abstract} We obtain results on cut and choose games for complete Boolean algebras. Zapletal proved that there is a Boolean algebra $\B$ such that $\G_\omega^\candc(\B)$, the version of the game which ends on the $\omega$'th round, is undetermined. We prove that, assuming the consistency of a proper class of supercompact cardinals, the limit version $\G^{\candc}_{< \lambda}(\B)$, in which there are $\lambda$-many rounds but no concluding round, is consistently determined for all complete Boolean algebras $\B$ and all successor cardinals $\lambda$. In particular, this answers a question of Zapletal \cite[Question 2]{Zapletal1995}. We also show that undetermined instances of the game $\G^\candc_\lambda(\B)$ follow from the approachability property, extending results of Dobrinen, and we prove that undetermined instances are compatible with $\MM^{++}$.\end{abstract}
 
 % The \emph{cut and choose game} was introduced by Ulam in 1964 as an example of a potential source of independence. Jech popularized the study of the cut and choose game $\G^{\candc}_\lambda(\B)$ for a complete Boolean algebra $\B$, where $\lambda$ indicates the length of the game. Zaplatel proved from $\ZFC$ that there is a complete Boolean algebra $\B$ such that $\G^{\candc}_\omega(\B)$ is undetermined, and Dobrinen proved under additional set-theoretic assumptions that $\G^{\textsf{\textup{c\&c}}}_\lambda(\B)$ may be undetermined for uncountable cardinals $\lambda$. 

\maketitle

\section{Introduction}

Ulam introduced the cut and choose game in order to explore the connections between abstract set theory and infinite constructions in physical theories \cite{Ulam1964}. The game starts with a set $X$ and two players: $\gcut$ and $\gchoose$ in modern terminology. First, $\gcut$ divides $X$ into two disjoint sets, and then $\gchoose$ selects one, and then $\gcut$ divides the chosen set. This goes on for countably many steps, and in the end, if the sets that $\gchoose$ has selected have an intersection containing at least two points, then $\gchoose$ wins; otherwise $\gcut$ wins.\footnote{Technically, Ulam gives the winning condition that the set is nonempty. However, this gives $\gchoose$ the winning strategy of selecting a point $z \in X$ ahead of time and always taking the choice containing $z$.} The question that Ulam raised is whether $\gchoose$ might have a winning strategy. Mycielski and Steinhaus had introduced the axiom of determinacy two years beforehand \cite{Mycielski-Steinhaus1962}, so it was natural to consider the existence of winning strategies. The question was intended as an example of a problem that is, ``rather easy to state'' while embodying ``the core properties of abstract sets'' \cite[page 347]{Ulam1964}. Ulam speculated that the answers to such questions may be independent, citing the work of Paul Cohen which at the time was yet unpublished.

The context for this paper is found in work of Jech, who studied the cut and choose game and some of its variants for complete Boolean algebras \cite{Jech1984} (see also \cite{Jech1978}). He found an interplay between winning strategies and distributivity properties of the Boolean algebras and formulated the basic results, for example showing that a complete Boolean algebra adds reals if and only if $\gcut$ has a winning strategy in the cut and choose game of length $\omega$. He asked whether there is a Boolean algebra for which the cut and choose game of length $\omega$ is undetermined, and was answered by Zapletal \cite{Zapletal1995}, who obtained a number of other crucial results.

Zapletal's result on the existence of an undetermined game left the natural question for cut and choose games of uncountable length. Questions of this form arise in other contexts, including the Mycielski games from which the axiom of determinacy is defined. Typically, the questions are more tractable for games of countable length, and less tractable for games of uncountable length (see \cite{Neeman2008}). In  \autoref{zapletal-sec}, we obtain a model in which a version of the cut and choose game of length $\lambda$ but without a $\lambda$'th round is determined of all successors $\lambda$. This was originally motivated by a question of Zapletal \cite[Question 2]{Zapletal1995}, the answer to which is implied by the theorem.

In \autoref{sec-andmore}, we consider variants of the cut and choose game of length $\lambda$ which end with a final round. During the 2000's, Dobrinen obtained extensive results giving hypotheses that imply that uncountable length cut and choose games can be undetermined \cite{Dobrinen2003,Dobrinen2007,Cummings-Dobrinen2007,Dobrinen2008}. In Section 3, we obtain results from weaker hypotheses on models in which uncountable cut and choose games are undetermined.

This paper will consider a number of variants of the cut and choose game which use additional parameters. We will try to balance thoroughness with readability.

% and also commented on the question directly, writing, ``We are particularly interested in finding a Boolean algebra in which [the cut and choose game of length $\kappa$] is undetermined in ZFC'' \cite[page 69]{Dobrinen2007}.

\subsection{Definitions and Notation}

The reader is assumed to be familiar with the basics of forcing and large cardinals (see e.g.\ \cite{Jech2003,Kunen2014}). Some additional definitions that are well-known but do not appear in introductory sources will be given throughout.

We should also comment on some particularities of notation. The notation $\lambda$-distributive indicates that functions with domain $\lambda$ are not added (see below). If $\dot{x}$ is a $\P$-name and $G$ is $\P$-generic over $V$, we will not emphasize that $x = \dot{x}_G$ when the notation makes this obvious. As for images of functions, we write $f[X]=\{f(z) \mid z \in X\}$.

For a Boolean algebra $\B$, we let $\wedge$ and $\vee$ be the meet and join operations, respectively, and we let $\sim b$ indicate the complement of $b$. We let $0_\B$ and $1_\B$ be the respective minimal and maximal elements of $\B$. We let $\B^+ := \{b \in \B \mid b \ne 0_\B\}$. If $b \wedge c = 0_\B$, we write $b \perp c$. Boolean algebras are endowed with a natural order $\le_\B$ where $b \le_\B c$ if and only if $c \wedge b = b$. We say that $\B$ is \emph{$<\nu$}-complete if for all $X \subseteq \B$ with $|X|<\nu$, $\bigvee X \in \B$ and $\bigwedge X \in \B$. We say that $\B$ is \emph{complete} if it is $<\!\nu$-complete for all $\nu$. See the Handbook of Boolean Algebras to refer to elementary concepts \cite{Koppelberg-HOBA}.

\begin{definition} Given a Boolean algebra $\B$ and $b \in \B$, we say that $\{b_0,b_1\}$ is a \emph{partition} of $b$ if $b_0 \vee b_1 = b$ and $b_0 \wedge b_1 = 0_\B$. More generally, we say that $\{b_i : i < \mu\}$ is a partition of $b$ if $b_i \wedge b_j = 0_\B$ for $i<j<\mu$ and $\bigvee_{i<\mu}b_i = b$.

We let $\prt_\mu(b)$ denote the set of partitions of $b$ of cardinality $\mu$, we let $\prt(b) = \prt_2(b)$, and we let $\prt(\B) = \{\prt(b) \mid b \in \B\}$.\end{definition}

Note that we can equivalently define partitions as maximal disjoint subsets of $\B^+$.

\begin{definition}[{\cite[Def.17]{Balcar-Simon-HOBA}}]
 Let $\B$ be a Boolean algebra. Let $\lambda$, $\mu$, and $\nu$ be cardinals such that $\lambda, \mu \geq \omega$ and $\nu \geq 2$.

 \begin{enumerate}

 \item We say that $\B$ is \emph{$(\lambda,\mu,<\!\nu)$-distributive} if for every sequence of partitions $X = \{X_\alpha \mid  \alpha \in \lambda\}$ such that $|X_\alpha| \leq \mu$ there is a dense subset $D \subseteq \B^+$ such that for all $\bar{b} \in D$ and $\alpha \in \lambda$, $|\{c \in X_\alpha \mid c \wedge \bar{b} \neq 0_\B\}| < \nu$.
  
\item If $\B$ is $(\lambda,\mu, 2)$-distributive, then we say that $\B$ is \emph{$(\lambda,\mu)$-distributive}.

\item We say that $\B$ is \emph{$\lambda$-distributive} if it is $(\lambda, \mu)$-distributive for any $\mu$.

\end{enumerate}
\end{definition}

%We stated the three-parameter version of distributivity to orient the reader, but this paper will only focus on two-parameter distributivity.

For complete Boolean algebras, we mention an equivalent statement that relates to the notion of distributivity as it may be more generally understood, but we will not use it for this paper.

\begin{fact} For all infinite cardinals $\lambda$ and all cardinals $\mu$ and $\nu$, the following are equivalent for all \emph{complete} Boolean algebras $\B$.

\begin{enumerate}

\item $\B$ is $(\lambda,\mu,<\nu)$-distributive,

\item for all index sets $X$ of cardinality $\le \lambda$ and $I$ of cardinality $\le \mu$, and all families $\{b_i^\xi : \xi \in X, i \in I\}$ of elements of $\B$,
\[
\bigwedge_{\xi \in X} \bigvee_{i \in I} b_i^\xi = \bigvee_{g:X \to [I]^{<\nu}} \bigwedge_{\xi \in X} \bigvee_{i \in f(\xi)} b_i^\xi.
\]
\end{enumerate}
\end{fact}

\begin{fact} For all infinite cardinals $\lambda$ and all cardinals $\mu$ and $\nu$, the following are equivalent for all Boolean algebras $\B$.

\begin{enumerate}

\item $\B$ is $(\lambda,\mu,<\nu)$-distributive,

\item if $G$ is $\B$-generic over $V$ and $f \in V[G]$ is function $f:\lambda \to \mu$, then there is a function $g \in V$ such that $g: \lambda \to [\mu]^{<\nu}$ and for all $\xi < \lambda$, $f(\xi) \in g(\xi)$.

\end{enumerate}
\end{fact}

Note that if $\nu =2$ and we are talking about $(\lambda,\mu)$-distributivity, then point \emph{(2)} reduces to the statement that $\B$ does not add new functions $f:\lambda \to \mu$.

Let us define the simplest version of the game informally. It is a game of perfect information between two players, $\gcut$ and $\gchoose$. The cut and choose game of length $\lambda$ begins with $\gcut$ making an opening move $\bar{b} \in \B$ and selecting a partition $\{b^0_0,b_0^1\} \in \prt(\bar{b})$. Then choose selects an element of the partition $b_0^{\epsilon_0}$ where $\epsilon_0 \in \{0,1\}$. Then $\gcut$ plays another partition, and so on. At limits, if there is a nonzero infimum of $\gchoose$'s choices, then the game proceeds with $\gcut$ partitioning the infimum. If the infimum is zero, $\gchoose$ loses. If $\lambda$-many rounds occur without $\gchoose$ losing, then $\gchoose$ wins at the end if $\bigwedge_{i<\lambda}b_\xi^{\epsilon_\xi} \ne 0_\B$.

\begin{center}
\def\arraystretch{1.25}
\begin{tabular}{c | c | c | c | c | c | c | c} 
  $\gcut$ & $\bar{b},\{b_0^0,b_0^1\}$ &  & $\{b_1^0,b_1^1\}$ & $\ldots$  & $\{b_\xi^0,b_\xi^1\}$ & &\\ 
 \hline
  $\gchoose$ &  & $b_0^{\epsilon_0}$ &  & $\ldots$ &  & $b_\xi^{\epsilon_\xi}$ & $\ldots$ 
\end{tabular}
\end{center}

Now let us present a more formal definition with more parameters---namely, a parameter that allows $\gcut$ to play larger partitions, and a parameter that allows $\gchoose$ to select multiple components of a partition. Moreover, we will state a definition in such generality that it can include the limit versions of the game, where a run of length $\lambda$ does not require a $\lambda$'th round.

\begin{definition} Fix a complete Boolean algebra $\B$. Let $\lambda,\mu,\nu$ be cardinals with $\mu \ge \nu \ge 2$. We describe \emph{the cut and choose game} $\G^{\mu,<\nu}_{<\lambda}(\B)$. This is a game of perfect information between two players, denoted $\gcut$ and $\gchoose$ .

\begin{enumerate}

\item A partial play of $\G^{\mu,<\nu}_{<\lambda}(\B)$ of length $\eta$ where $\eta+1 < \lambda$ takes the form
\[
\vec{s} = \bar{b} {}^\frown \seq{\{b_i^\xi : i< \mu \},B_\xi}{\xi < \eta}
\]
where the $B_\xi \in [\mu]^{<\nu}$ for all $\xi < \eta$ if it is $\gcut$'s turn, or
\[
\vec{s} = \bar{b} {}^\frown \seq{\{b_i^\xi : i< \mu \},B_\xi}{\xi < \eta} {}^\frown \{b_i^\eta : i< \mu \}
\]
if it is $\gchoose$'s turn,
\[
\{b^\xi_i: i < \mu \} \in \prt_\mu \left(\bigwedge_{\zeta<\xi}\bigvee_{i \in B_\zeta} b_i^\zeta \right)
\]
for all $\xi \le \eta$ where we are inductively assuming that 
\[
0_\B \ne \bigwedge_{\zeta<\xi}\bigvee_{i \in B_\zeta}b_i^\zeta,
\]
and $B_\xi \in [\mu]^{<\nu}$ for all $\xi \le \eta$. We refer to $\bar{b}$ as the \emph{opening move}, we refer to $\{b_i^\xi:i<\mu\}$, $Z_\xi$ as the \emph{$\xi$'th round}.

\item The player $\gchoose$ wins the game if, for all plays,
\[
\bar{b} \wedge \bigwedge_{\xi<\eta} \bigvee_{i \in B_\xi} b_i^\xi \ne 0_\B
\]
where $\bar{b}$ is the opening move.
\end{enumerate}

We also employ the following conventions:

\begin{itemize}[$\circ$]
\item We let $\G_\lambda^{\mu,\nu}(\B)$ denote $\G^{\mu,<(\nu+1)}_{<(\lambda+1)}(\B)$, and we handle analogous parameterizations with with non-strict inequalities similarly.
\item We let $\G_{<\lambda}^{<\mu}(\B)$ denote $\G_{<\lambda}^{\mu,1}(\B)$. 
\item We let $\G_\lambda^\candc(\B)$ denote $\G^2_\lambda(\B)$.
\item There is an analogous game denoted $\G^{<\mu,<\nu}_{<\lambda}(\B)$, in which $\gcut$ plays partitions of cardinality less than $\mu$, but this is slightly more notationally cumbersome and does not appear in this paper.
\end{itemize}
\end{definition}

In our notation, $\G_{\omega}^2(\B)$ is the standard cut and choose game of length $\omega$.

\begin{fact}\label{two-games} Let $\B$ be a complete Boolean algebra and $\lambda$ a cardinal. The existence winning strategy for $\gcut$ or, respectively, $\gchoose$, is equivalent for the following games:

\begin{enumerate}

\item $\G_{<\lambda}^{\mu,<\nu}(\B)$,

\item an alternate version of $\G_{<\lambda}^{\mu,<\nu}(\B)$ which is identical except that partial plays take the form 
\[
\vec{s} = \bar{b} {}^\frown \seq{\{b_i^\xi : i< \mu \},B_\xi}{\xi < \eta}
\]
where
\[
\{b_\xi^i : i < \mu\} \in \prt(1_\B).
\]
and $B_\xi \in [\mu]^{<\nu}$ for all $\xi < \eta$, and the winning conduction requires that
\[
\bar{b} \wedge \bigwedge_{\xi<\eta} \bigvee_{i \in B_\xi}b^\xi_i \ne 0_\B.
\]
\end{enumerate}
\end{fact}

In other words, the cut and choose game is equivalent to the version where $\gcut$ only partitions $1_\B$ rather than partitioning $\gchoose$'s selections.

Sometimes, the notation $\G^\infty_\lambda(\B)$ is used to indicate that $\gcut$ may play partitions of arbitrary cardinality. The Banach-Mazur game of length $\kappa$ on Boolean algebras can be written as $\G^\infty_\kappa(\B)$ for the given Boolean algebra $\B$ because work of Jech  and Veli{\v c}kovi{\'c} (\cite[Theorem 2]{Jech1984} and \cite[Theorem 2.1]{Velickovic1986}) shows that there is an equivalence.

\section{Extending a Theorem of Zapletal}\label{zapletal-sec}

\subsection{Two-Parameter Strong Distributivity and the Existence of Winning Strategies}

For this section we want to establish an approach for studying the existence  or non-existence of winning strategies in the cut and choose game.

For purposes of discussion and later use, we recall the definition of strategic closure for posets (see \cite[Section 6]{Handbook-Cummings}).

\begin{definition} Fix a poset $\P$. We define the \emph{Banach-Mazur game of length $\lambda$} on $\P$.

The game is a game of perfect information between two players. The play of the game is a decreasing sequence of conditions $\seq{p_\xi}{\xi<\eta}$ in which Player \textsf{II} chooses conditions $p_\xi$ at limits $\xi$ as well as even successor ordinals (successors of the form $\xi+n$ where $\xi$ is a limit and $n$ is even), and Player \textsf{I} chooses conditions at odd successor ordinals. Player \textsf{II} wins a play $\seq{p_\xi}{\xi < \eta}$ of length $\eta$ if it is possible for the players to make a move at every step of the play, i.e.\ there is a lower bound of the sequence $\seq{p_\xi}{\xi<\xi'}$ for all $\xi' < \eta$.

We say that $\P$ is $\eta$\emph{-weakly strategically closed}, or simply that $\P$ is $\eta$\emph{-strategically closed}, if Player \textsf{II} has a winning strategy for all games of length $\eta$. In other words, $\P$ is $\eta$\emph{-strategically closed} if the there is a function $\sigma: {}^{<\eta} \P \to \P$ such that for all $\xi < \eta$, if $p_\xi = \sigma(\seq{p_\zeta}{\zeta<\xi})$ for all $\xi<\eta$ such that $\xi$ is either a limit or an even successor, then $\seq{p_\zeta}{\zeta<\xi}$ has a lower bound.\end{definition}

The notion of \emph{strong distributivity} was introduced by Jakob \cite[Definition 3.1]{Jakob2025}, who observed that it characterizes existence of a winning strategy for the adversary in the Banach-Mazur game, as illustrated in a theorem of Foreman, which generalized a theorem of Jech:

\begin{fact} Player $\mathsf{I}$ has a winning strategy in the Banach-mazur game of length $\lambda+1$ on $\B$ if and only if $\B$ is not $\le \lambda$-distributive \cite[page 718]{Foreman1983}.\end{fact}

Jakob noted that this characterization can be generalized and found a number of applications (see also \cite{Jakob2026}). We define the natural two- and three-parameter versions of strong distributivity.

\begin{definition} Let $\B$ be a Boolean algebra, let $\gamma$ be an ordinal, and let $\mu$ be a cardinal.

\begin{enumerate}
\item We say that $\B$ is \emph{strongly $(<\!\gamma,\mu,<\nu)$-distributive} if for all $\bar{b} \in \B$ and all sequences $\seq{\{b^\xi_i:i<\mu\}}{\xi<\gamma}$ of partitions of $\B$ compatible with $\bar{b}$, there is a $\le_\B$-decreasing sequence $\seq{c_\xi}{\xi<\gamma}$ such that for all $\xi<\gamma$, $|\{i < \mu  \mid c_\xi \wedge b_i^\xi \ne 0_\B \}| < \nu$.
\item $\B$ is \emph{strongly $(<\!\gamma,\mu)$-distributive} if $\nu = 2$.
\end{enumerate}\end{definition}

Dobrinen investigated the relationship between winning strategies and three-parameter distributivity. The question of equivalence between three-parameter distributivity and non-existence of a winning strategy for $\gcut$ in the appropriate game is complex, but one direction is relatively simple (see \cite[Theorem 2.2]{Dobrinen2003}).

\begin{lemma}\label{3par-dist-ws} Let $\B$ be a complete Boolean algebra, let $\lambda$ and $\mu$ be cardinals, and suppose that $\B$ is not strongly $(<\!\lambda,\mu,<\nu)$-distributive. Then $\gcut$ has a winning strategy in $\G^{\mu,<\nu}_{<\lambda}(\B)$.\end{lemma}

\begin{proof} Let $\bar{b} \in \B^+$ and let $\seq{\{b_i^\xi:i<\mu \}}{\xi<\lambda}$ be a sequence of partitions witnessing the failure of strong $(<\!\lambda,\mu,<\!\nu)$-distributivity, meaning that there is no $\le_\B$-decreasing sequence $\seq{c_\xi}{\xi<\lambda}$ below $\bar{b}$ such that  $c_\xi$ is compatible with $b_i^\xi$ for fewer than $\nu$-many indices for all $\xi<\lambda$. Then the winning strategy for $\gcut$ is to open with $\bar{b}$ and to play $\{b_i^\xi:i<\mu\}$ at the $\xi$'th round. Suppose for contradiction that $\gchoose$ has a successful play in the form of $f:\lambda \to [\mu]^{<\nu}$ choosing indices $i$ at the $\xi$'th round. Then for each $\eta<\lambda$, we define $c_\eta$ where
\[
0_\B \ne c_\eta := \bar{b} \wedge \bigwedge_{\xi<\eta} \bigvee_{i \in f(\xi)}b_i^\xi,
\]
and in particular $\seq{c_\eta}{\eta<\lambda}$ is $\le_\B$-decreasing. But then it is in fact the case that $b$ is compatible with $b_i^\xi$ only for indices in the set $f(\xi)$ for each $\xi<\lambda$, a contradiction.\end{proof}

Nonexistence of a winning strategy for $\gcut$ will give us an opening to apply the ideas from \autoref{zapletal-answer}. Here is a variation of lemmas found in Jech \cite{Jech1984}, Foreman \cite{Foreman1983}, Dobrinen \cite{Dobrinen2003} (who notes certain limitations when full completeness is not assumed), and others:

\begin{lemma}\label{foreman-generalization} Suppose that $\B$ is a complete Boolean algebra and let $\gamma$ be an ordinal and let $\mu$ be a cardinal. Then the following are equivalent:
\begin{enumerate}
\item $\gcut$ does not have a winning strategy in $\G_{<\gamma}^\mu(\B)$,
\item $\B$ is strongly $(<\!\gamma,\mu)$-distributive. 
\end{enumerate}\end{lemma}

\begin{proof} To prove the forward direction, just apply \autoref{3par-dist-ws} for $\nu = 2$.

%suppose that $\B$ is not strongly $(<\!\gamma,\mu)$-distributive. Then let $\seq{\{b_\xi^i : i< \mu\}}{\xi<\gamma}$ be a sequence of partitions compatible with some $\bar{b}$ witnessing this. Then $\gcut$ wins precisely by playing $\bar{b}$, and then $\{b_\xi^i : i < \mu\}$ at the $\xi$'th step. Suppose for contradiction that $\gchoose$ wins the play by selecting $b_\xi^{i_\xi}$ at every stage: Then for all $\xi<\gamma$, let $c_\xi := \bar{b} \wedge \bigwedge_{\zeta<\xi}b_\xi^{i\xi} \ne 0_\B$; then the sequence $\seq{c_\xi}{\xi<\gamma}$, which is necessarily $\le_\B$-decreasing, contradicts the fact that $\bar{b}$ and $\seq{\{b_\xi^i : i<\mu \}}{\xi<\gamma}$ together witness the failure of strong $(<\!\gamma,\mu)$-distributivity.

For the other direction, suppose that $\gcut$ has a winning strategy $\sigma$ with first move $\bar{b}$, and suppose for contradiction that $\B$ is $(<\!\gamma,\mu,<\nu)$-distributive. We construct a tree $T \subseteq {}^{<\gamma}\mu$ of height $\gamma$ and an assignment $t \mapsto \{b^i_t : i < \mu\}$ such that:

\begin{itemize}[$\circ$]
\item if $u = t {}^\frown \langle j \rangle$, then $\{b_i^u : i < \mu\} \in \prt_\mu(b^t_j)$,
\item letting $\vec{s}_t$ denote the partial play of the game of length $\dom(t) \cdot 2 +1$, $\{b_i^t:i<\mu\}$ is obtained by applying $\sigma$ to $\vec{s}$.
\end{itemize}

We define $T$ be induction on levels. Let $\emptyset \in T$, and if $t \in T$, let $\{t {}^\frown \langle i \rangle \mid i < \mu \} \subseteq T$. Now suppose $\xi$ is a limit and we have defined $T_\zeta$ for $\zeta<\xi$. We let $T_\xi$ be the set of $t$ cofinal in $\bigcup_{\zeta<\xi}T_\xi$ such that
\[
\bigwedge_{\zeta<\xi}b_{t \rest \zeta}^{t(\zeta)} \ne 0_\B.
\]
By $(|\xi|,\mu)$-distributivity (we do not technically need strong $(<\!\gamma,2)$-distributivity for this step), $T_\xi$ is non-empty. Lastly, let $T = \bigcup_{\xi<\gamma}T_\xi$.

Having constructed $T$, let
\[
c_\xi^i = \bigvee_{t \in T_\xi} b_i^t
\]
for $\xi<\gamma$ and $i<\mu$. Let $\vec{d}:=\seq{d_\xi}{\xi < \gamma}$ witness strong $(<\!\gamma,\mu)$-distributivity with respect to $\seq{\{c_\xi^i:i<\mu\}}{\xi<\gamma}$ and $\bar{b}$. Then $\vec{d}$ determines a cofinal branch through $T$: By induction on $\xi<\gamma$, build a sequence of nodes $\seq{t_\xi}{\xi<\gamma}$ through $T$ so that if $t_\xi$ has been defined, then we choose $t_{\xi+1}$ so that $t_{\xi+1} = t_\xi {}^\frown \langle i \rangle$ where $d_{\xi+1} \le_\B c_i^\xi$. But then the play given by the branch is winning for $\gchoose$, so we have a contradiction of the assertion that $\sigma$ is a winning strategy for $\gcut$.\end{proof}

\subsection{Simple Instances of Determinacy}

For the sake of completeness, we include a proposition that is probably in the folklore. It allows us to state our results in a bit more generality.

 \begin{proposition}\label{shortbad} If $\B$ is a complete Boolean algebra and $\lambda,\mu,\nu$ are cardinals such that $|\B| \le \lambda$, then $\G^{\mu,<\nu}_{\lambda}(\B)$ is determined.\end{proposition}

\begin{proof}[Proof of \autoref{shortbad}] If the set of atoms is not dense in $\B$, then $\gcut$ has a winning strategy. They open with some $b \in \B$ below which there are no atoms. Then, since $|\B \rest b| \le \lambda$, it follows that $\B$ not $\le \lambda$-distributive, so the statement follows from \autoref{foreman-generalization}. If the set of atoms is dense in $\B$, then $\gchoose$ has a winning strategy: If cut opens with $\bar{b}$, then fix some atom $a \le_\B \bar{b}$. Then for every partition $\{b_i : i< \mu\}$ that $\gcut$ plays, there is a unique $i<\mu$ such that $a \le_\B b_i$, so $\gchoose$ selects a set containing whichever element is above $a$.\end{proof}

Let us clarify a definition that might normally be taken for granted. 

\begin{definition} Let $\B$ be a $<\lambda$-complete Boolean algebra.

\begin{enumerate}

\item We say that $\mathcal{F} \subseteq \mathcal{P}(\B)$ is an \emph{ultrafilter} if the following hold: for all $b \in \B$, either $b \in \mathcal{F}$ or $\sim b \in \mathcal{F}$; for all $b,c \in \mathcal{F}$, $b \wedge c \in \mathcal{F}$; $b \in \mathcal{F}$ and $c \ge_\B b$ implies $c \in \mathcal{F}$; and $0_\B \notin \mathcal{F}$.

\item Let $\mathcal{F}$ be an ultrafilter on a complete Boolean algebra $\B$ and let $\lambda$ be a regular cardinal. We say that $\mathcal{F}$ is $<\lambda$-complete if for all $\tau<\lambda$ and $\seq{b_\xi}{\xi<\tau} \subseteq \mathcal{F}$, we have $\bigwedge_{\xi<\tau}b_\xi \in \mathcal{F}$.\end{enumerate}\end{definition}

\begin{proposition}\label{easy-filter} Suppose that $\B$ is a complete Boolean algebra, $\lambda$ is a regular cardinal, and that for all $\bar{b} \in \B^+$ there is a $<\!\lambda$-complete ultrafilter $\mathcal{F}$ on $\B$ with $\bar{b} \in \F$. Then it follows that for all $\nu \le \mu < \lambda$, $\gchoose$ has a winning strategy in the game $\G_{<\lambda}^{\mu,<\nu}(\B)$.\end{proposition}

\begin{proof} First, note that as with any $<\!\lambda$-complete ultrafilter $\F$ in other settings, if $\{b_i : i < \mu\}$ is a partition of $\B$, then there is some unique $i<\mu$ such that $b_i \in \F$. 

We define a winning strategy $\sigma$ for $\gchoose$ with respect to an opening move $\bar{b}$ played by $\gcut$. Let $\mathcal{F}$ be a $<\!\lambda$-complete ultrafilter such that $\bar{b} \in \mathcal{F}$.

Suppose $\vec{s}$ is a partial play at the $\xi$'th round where it is $\gchoose$'s turn and the last move of $\gcut$ is $\{b^\xi_i : i < \mu\}$. Then $\gchoose$ selects a set containing $i_\xi$ where $b^\xi_{i_\xi} \in \mathcal{F}$. We argue that $\sigma$ is well-defined and is a winning strategy by observing that for all partial plays $\vec{s}$, if $\{b_{i_\xi}^\xi \mid \xi<\eta\}$ is the set of $\gchoose$'s choices, then $b_* := \bigwedge_{\xi<\eta}b_{i_\xi}^\xi$ is in $\mathcal{F}$ and therefore is not equal to $0_\B$. Therefore the play can continue.\end{proof}
 
Now we can obtain some simple determinacy results for $\G_{<\lambda}^\mu(\B)$.

\begin{proposition}\label{easy-determinacy} If $\lambda$ is a regular cardinal and $\B$ is a complete Boolean algebra of cardinality $\le \lambda$, then $\G_{<\lambda}^{\mu}(\B)$ is determined for all $ \mu<\lambda$.\end{proposition}

\begin{proof} \autoref{shortbad} tells us to only consider $\B$ of cardinality $\lambda$. Fix some $\bar{b} \in \B$ and let $\seq{\{b_\xi,\sim b_\xi\} }{\xi<\lambda}$ enumerate the binary partitions of $\B$. If $\gcut$ has does not have a winning strategy, then \autoref{foreman-generalization} implies that there is a $\le_\B$-decreasing sequence $\seq{c_\xi}{\xi<\lambda}$ through $\B$ such that for all $\xi<\lambda$, there is a unique $\epsilon \in \{0,1\}$ such that $c_\xi \le b^\xi_\epsilon$ where $b^\xi_0 = b_\xi$ and $b^\xi_1 = \sim b_\xi$. Then it is clear that
\[
\mathcal{F}_{\bar b} := \{ b \in \B \mid \exists \xi<\lambda, c_\xi \le_\B b \}
\]
is a $<\!\lambda$-complete ultrafilter on $\B$ containing $\bar{b}$, so the statement follows from \autoref{easy-filter}.\end{proof}

%This is a good moment to sketch some possible extensions of \autoref{zapletal-answer}. If we remove the forcing from \autoref{zapletal-answer}, we essentially obtain the following theorem.

\begin{theorem}\label{from-a-sc-only} If $\kappa$ is a supercompact cardinal and $\lambda \ge \kappa > \mu$ where $\lambda$ is regular, then $\G^{\mu}_{<\lambda}(\B)$ is determined for all complete Boolean algebras $\B$.\end{theorem}

\begin{proof} The forward direction of \autoref{foreman-generalization} reduces the problem to arguing that if $\B$ is a strongly $(<\!\kappa,\mu)$-distributive Boolean algebra of cardinality $\kappa$, then $\gchoose$ has a winning strategy in $\G^{\mu}_{<\kappa}(\B)$. Fix a $\lambda $-supercompact embedding $j:V \to M$ with critical point $\kappa$.

Fix an arbitrary $\bar{b} \in \B$. Then we can argue that we can define a $<\!\lambda$-complete $\B$-ultrafilter $\F$, from which \autoref{easy-filter} gives us a winning strategy for $\gchoose$

In order to define this $\F$, let $\vec{b}=\seq{\{b_\xi,\sim b_\xi\}}{\xi<\lambda}$ enumerate the binary partitions of $\B$ compatible with $\bar{b}$. By elementarity, $M \models \bq j(\B)$ is $(<\!j(\kappa),\mu)$-distributive$\eq$, and in particular $M \models \bq j(\B)$ is $(<j(\kappa),2)$-distributive$\eq$, so given that $\seq{\{j(b_\xi),\sim j(b_\xi)\}}{\xi<\lambda}$ is a sequence of partitions of $j(\B)$ in $M$,  there is some $b_{\textup{gen}} \in M \cap j(\B)$ such that for all $\xi<\lambda$, either $b_{\textup{gen}} \le_{j(\B)} j(b_\xi)$ or $b_\textup{gen} \le_{j(\B)} (\sim j(b_\xi))$. 

%Let $\vec{b}=\seq{\{b^\xi_i : i<\mu \}}{\xi<\lambda}$ enumerate the $\mu$-sized partitions of $\B$. By elementarity, $M \models \bq j(\B)$ is $(<\!j(\kappa),\mu)$-distributive$\eq$, and in particular $M \models \bq j(\B)$ is $(\kappa,\mu)$-distributive$\eq$, so given that $\seq{\{j(b_i^\xi):i<\mu\}}{\xi<\lambda}$ is a sequence of partitions of $j(\B)$ in $M$,  there is some $b_{\textup{gen}} \in M \cap j(\B)$ such that for all $\xi<\lambda$, there is a unique $i<\mu$ such that $b_{\textup{gen}} \le_{j(\B)} j(b_i^\xi)$. 

Then we can argue that we can define a $<\!\lambda$-complete $\B$-ultrafilter $\F$ from $b_\textup{gen}$, from which \autoref{easy-filter} gives us a winning strategy for $\gchoose$: We let
\[
\F = \{ c \in \B : b_\textup{gen} \le_{j(\B)} c \}.
\]
It is immediate that $\F$ is a $\B$-ultrafilter such that $\bar{b} \in \mathcal{F}$. Suppose that $\{c_i : i<\tau \} \subseteq \F$ for some $\tau<\kappa$. Then we have
\[
0_{j(\B)} \ne b_\textup{gen} \le_{j(\B)} \bigwedge_{i<\tau}j(b_i) = j \left( \bigwedge_{i<\tau}b_i \right)
\]
so by elementarity it is the case that $\bigwedge_{i<\tau}b_i $ is nonzero.\end{proof}

\subsection{Lifting Embeddings and a Question of Zapletal}

Zapletal proved that, assuming the consistency of a supercompact cardinal, there is a model of set theory in which $\gchoose$ has a winning strategy in $\G^{\candc}_\omega(\B)$ for every $\aleph_1$-distributive Boolean algebra $\B$.

We want to obtain an improvement:

\begin{theorem}\label{zapletal-answer} Suppose $\kappa$ is a supercompact cardinal. Then if $G$ is $\Col(\omega_1,<\!\kappa)$-generic over $V$, then in $V[G]$, for every Boolean algebra $\B$ that does not add new subsets of $\omega_1$, the player $\gchoose$ has a winning strategy in $\G_\omega^\candc(\B)$.\end{theorem}

This answers Question 2 in Zapletal's paper ``More on the Cut and Choose Game'' \cite{Zapletal1995}. In the same paper, Zapletal obtained a large cardinal lower bound that applies to this theorem \cite[Example 3]{Zapletal1995}.

In fact, we will have an even stronger statement in the form of determinacy for the limit version of the game:

\begin{theorem}\label{zapletal-actually} If $\kappa$ is a supercompact cardinal, $\lambda<\kappa$ is regular, and $G$ is $\Col(\lambda,<\!\kappa)$-generic over $V$, then in $V[G]$ it holds that for every Boolean algebra $\B$, $\G_{<\lambda}^{\mu}(\B)$ is determined for all cardinals $\mu < \lambda$.\end{theorem}

We prove our main preservation lemma.

\begin{lemma}\label{closure-preservation1} Suppose $\P$ is a $<\!\lambda$-closed poset that preserves a cardinal $\mu$ and that $\B$ is a complete Boolean algebra. If $\P$ forces that $\gchoose$ has a winning strategy for $\G_{<\lambda}^\mu(\B)$, then $\gchoose$ has a winning strategy for $\G_{<\lambda}^\mu(\B)$ in the ground model.\end{lemma}

\begin{proof} Let $\dot{\sigma}$ be a $\P$-name that is forced to be a winning strategy for $\gchoose$ in $\G_{<\lambda}^\mu(\B)$. We fix some $\bar{b}$ and describe a winning strategy for a run of the game where $\bar{b}$ is $\gcut$'s opening move. Let $\Theta = |\B|$, and for $b \in \B$, let $\seq{\{b^\alpha_i : i < \mu\}}{\alpha<\Theta}$ be a fixed enumeration of the $\mu$-sized partitions of $\B$ that are compatible with $\bar{b}$ (allowing repetitions if necessary). 

We will construct an assignment on the tree ${}^{<\lambda} \Theta$ of the form $t \mapsto (b^{\bar{b}}_t,p^{\bar{b}}_t) \in \B^+ \times \P$ for $t \in {}^{<\lambda}\Theta$. Technically the superscripts give information about the opening move, but let us suppress this notation and instead write $(b_t,p_t)$ for $t \in {}^{<\lambda}\Theta$.

This assignment will have the property that $b_{t {}^\frown \langle \alpha \rangle} \in \{b_t^{\alpha,i}: i < \mu \}$ for all $t$. Therefore, an element $t$ with domain $\dom(t)=\zeta$ can be associated with a play of $\G_{<\lambda}^\mu(\B)$ of the form
\[
\vec{s}_t = \bar{b} {}^\frown \seq{ \{b_{t \rest i}^{t(\xi),i}: i < \mu\},b_{t(\xi)}}{\xi<\zeta}
\] 
where it is $\gcut$'s turn to play.

We will construct the assignment such that the following properties hold:

\begin{enumerate}

\item $b_\emptyset = \bar{b}$,

\item if $u \sqsupseteq t$ then $p_u \le_\P p_t$,

\item if $\dom(t)=\xi+1$ and $t = u {}^\frown \langle \alpha \rangle$, then $p_t \Vdash \bq \dot{\sigma}(\vec{s}_u {}^\frown \{b_u^{\alpha,i}: i < \mu \}) = b_t\eq$.

\end{enumerate}

The assignment is constructed by induction on $\dom(t)$. 

Suppose that $\dom(t)=\xi+1$, that $t = u {}^\frown \langle \alpha \rangle$, and we have defined $b_u$ and $p_u$. Then choose $p' \le p_u$ such that $p'$ forces $``\dot{\sigma}(\bar{b} {}^\frown  \vec{s}_u {}^\frown \{b_u^{\alpha,i}: i < \mu\})= b'\textup{''}$. Then let $p_t = p'$ and let $b_t = b'$.

Suppose that $\dom(t)$ is a limit $\xi$. Then let $p_t$ be a lower bound of $\seq{p_{t \rest \xi}}{\xi<\dom(t)}$ and let $b_t = \bigwedge_{\xi<\dom(t)}b_{t \rest \xi}$. Since $p_t$ forces that $\vec{s}_t$ is a run of the game in which $\gchoose$ plays according to $\dot{\sigma}$, and that $b_t$ is the meet of $\gchoose$'s selections, it follows that $b_t \ne 0_\B$.

Now we will define a winning strategy $\tilde{\sigma}$ for $\gchoose$ in $V$. Fix an opening move $\bar{b}$; we will denote $\tilde{\sigma}$ with respect to this opening move given the assignment $t \mapsto (b_t,p_t)$ with $\bar{b}$ implicit.  In particular, we will define $\tilde{\sigma}$ such that any play of the game (with $\bar{b}$ opening) in which it is $\gchoose$'s turn will take the form $\bar{b} {}^\frown  \vec{s}_t {}^\frown \{b_t^{\alpha,i}: i < \mu \}$ for some $t \in {}^{<\lambda}\Theta$ and $\alpha<\Theta$. Given such a position, $\gchoose$ will select $b_{t {}^\frown \langle \alpha \rangle}$.

To see that this is a winning strategy, note that full run of the game in which $\gchoose$ plays according to $\tilde{\sigma}$, minus $\gchoose$'s last move, can be represented as
\[
\vec{s}=\bar{b} {}^\frown  \seq{\{b_{f \rest \xi}^{f(\xi),i}: i < \mu \},b_{f(\xi)}}{\xi<\lambda}
\]
for some $f:\lambda \to \Theta$. The play continues at every stage because of the way the tree is defined\end{proof}

% Let $p_f$ be a lower bound of $\seq{p_{f \rest \xi}}{\xi<\lambda}$. Then $p_f$ forces that $\vec{s}$ is a run of the game in which $\gchoose$ plays according to the strategy $\dot{\sigma}$, and therefore an extension of $p_f$ forces that $\seq{b_{f \rest \xi}}{\xi<\lambda}$ has a nonzero lower bound $b_f$. Hence $\seq{b_{f \rest \xi}}{\xi<\lambda}$ has a lower bound in $V$.

\begin{proof}[Proof of \autoref{zapletal-answer}]

We will show that in $V[G]$, for any complete and strongly $(<\!\lambda,\mu)$-distributive Boolean algebra $\B$, there is a winning strategy in the game $\G^{\mu}_{<\lambda}(\B)$ for all cardinals $ \mu < \lambda$. This is sufficient by \autoref{3par-dist-ws}.

 Let $\dot{\B}$ be a $\Col(\lambda,<\!\kappa)$-name for such a Boolean algebra of cardinality $\chi$, let $j:V \to M$ be a $\chi$-supercompact embedding with critical point $\kappa$, and let $G$ be $\Col(\lambda,<\!\kappa)$-generic over $V$. Write $\B = \dot{\B}_G$.
 
As in standard arguments, we write
\[
j(\Col(\lambda,<\!\kappa))=\Col(\lambda,<\!\kappa) \times \prod_{\kappa \le \alpha < j(\kappa)}\Col(\lambda,\alpha) = \Col(\lambda,<\!\kappa) \times \P
\]
where $\P$ is our notation for the $<\lambda$-closed remainder term.

We will argue that in $M[G][H]$, there is a $<\lambda$-complete ultrafilter $\mathcal{F}$ on $\B$. By \autoref{easy-filter}, this implies that in $V[G][H]$, $\gchoose$ has a winning strategy in $\G_{<\lambda}^{\mu}(\B)$. By \autoref{closure-preservation1} and the fact that $\P$ is $\lambda$-closed, this implies that $\gchoose$ has a winning strategy for $\G_{<\lambda}^{\mu}(\B)$ in $V[G]$.

We establish some conditions for defining a generic element of $j(\B)$. Let $\seq{\dot{b}_\xi}{\xi<\chi}$ be a sequence of $\Col(\lambda,<\!\kappa)$-names forced to be an enumeration of $\dot{\B}$. Then $\seq{j(\dot{b}_\xi)}{\xi<\chi} \in M$ by $M^\chi \subseteq M$, and since $G \ast H \in M[G \ast H]$, it follows that we have $\seq{j(b_\xi)}{\xi<\chi} \in M[G][H]$ where $b_\xi = (\dot{b}_\xi)_G$. Then $\seq{\{ j(b_\xi),\sim j(b_\xi)\}}{\xi<\chi} \in M[G][H]$ is a sequence of partitions of $j(\B)$. Also, in $M[G][H]$, $\chi$ has cardinality $\lambda$: if $D_\alpha \subseteq j(\Col(\lambda,<\!\kappa))$ is the open dense subset of conditions $p$ with $\langle i, \chi \rangle \in \dom(p)$ and $p(i,\chi) = \alpha$ for some $i<\lambda$ (using that $j(\kappa)>\chi$), then $D_\alpha^M = D_\alpha^V$ and $D_\alpha \in M$ for all $\alpha<\chi$, so $\seq{D_\alpha}{\alpha<\chi} \in M$, so we can define a surjection $\lambda \to \chi$ in $M[G][H]$ by choosing $p_\alpha \in D_\alpha \cap (G \ast H)$ and taking $\bigcup_{\alpha \in \chi}p_\alpha$.  Also, $M[G][H] \models \textup{``}j(\B)$ is strongly $(<\!j(\lambda),j(2))=(<\!j(\lambda),2)$-distributive$\textup{''}$ by elementarity.

Fix some $\bar{b} \in \B$. Now we will define the object from which we define a $\lambda$-complete $\B$-ultrafilter containing $\bar{b}$. Let $\Phi:\lambda \to \chi$ be a surjection in $M[G][H]$. Applying strong $(<\!\lambda,2)$-distributivity in $M[G][H]$, we find a $\le_{j(\B)}$-decreasing sequence $\seq{c_\alpha^\textup{gen}}{\alpha<\lambda} \in j(\B)^+ \cap M[G][H]$ below $j(\bar{b})$ such that for all $\alpha<\chi$, $c_\alpha^\textup{gen} \wedge d \ne 0_\B$ for exactly one of $d \in \{j(b_{\Phi(\alpha)}),\sim \! j(b_{\Phi(\alpha)})\}$; in other words, either $c_\alpha^\textup{gen} \le_{j(\B)} j(b_{\Phi(\alpha)})$ or $c_\xi^\textup{gen} \perp_{j(\B)} j(b_{\Phi(\alpha)})$.

We can therefore define
\[
\mathcal{F} = \{b \in \B \mid \exists \alpha<\lambda, c_\alpha^\textup{gen} \le_{j(\B)} j(b)\}
\]
in $V[G][H]$.

That $\mathcal{F}$ is a filter is immediate from elementarity of $j$. To see that $\mathcal{F}$ is an ultrafilter, suppose that $\{c_0,c_1\}$ partitions $\B$. Then let $\xi<\chi$ be such that $c_0=b_\xi$ and $c_1 = \sim \! b_\xi$ and let $\alpha<\lambda$ be such that $\Phi(\alpha) = \xi$. Then $\{j(b_\xi),\sim j(b_\xi)\}$ partitions $j(\B)$ in $M[G][H]$, so either $b_\textup{gen} \wedge j(b_\xi) \ne 0_{j(\B)}$ or $b_\textup{gen} \wedge \sim j(b_\xi) \ne 0_{j(\B)}$. Hence $b_\textup{gen}$ is $\le_{j(\B)}$-below either $j(b_\xi)$ or $\sim j(b_\xi)$.

To see that $\mathcal{F}$ is $<\lambda$-complete, consider a sequence $\seq{d_i}{i<\tau} \subseteq \mathcal{F}$ with $\tau<\lambda$. For $i<\tau$, let $\xi_i$ be such that either $d_i = b_{\xi_i}$ or $d_i = \sim \! b_{\xi_i}$ and let $\alpha_i$ be such that $\Phi(\alpha_i) = \xi_i$.  If $\alpha_* = \sup_{i<\tau}\alpha_i < \lambda$, we have that $c^\textup{gen}_{\alpha_*} \le d_i$ for all $i<\tau$. Since $\P$ is $<\!\lambda$-closed, the sequence $\vec{x}:=\seq{d_i}{i<\tau}$ is an element of $V[G]$. Therefore $M[G][H] \models \textup{``}j(\vec{x})$ has a greatest lower bound in $j(\B)$ above $c^\textup{gen}_{\alpha_*}\textup{''}$, since $j(d_i) \ge_{j(\B)} c^\textup{gen}_{\alpha_*}$ for all $i<\tau$. Therefore, by elementarity, $V[G] \models \textup{``}\vec{x}$ has a nonzero greatest lower bound in $\B\textup{''}$. Let $d_*$ be the greatest lower bound of $\vec{x}$ in $\B$ as computed in $V[G]$. Then by elementarity, since $d_* \le_{j(\B)} j(d_{\xi_i})$ for all $i<\tau$, it follows that
\[
c^\textup{gen}_{\alpha_*} \le \bigwedge_{i<\chi}j(d_{\xi_i}) = j \left(\bigwedge_{i<\chi}j(d_{\xi_i}) \right) = j(d_*),
\]
and so $d_* \in \mathcal{F}$.

As indicated, the proof is finished by the fact that $\mathcal{F}$ is a $<\!\lambda$-complete complete ultrafilter on $\B$ that contains $\bar{b}$ as an element and is defined in $V[G][H]$.\end{proof}

%\subsection{A Global Version of the Theorem} 

It is fairly straightforward to obtain a global version of \autoref{zapletal-answer}.

\begin{theorem}\label{global-zapletal} Assume the consistency of a proper class of supercompact cardinals. Then it is consistent that, for all successor cardinals $\lambda=\tau^+$ and every complete Boolean algebra $\B$, the game $\G^{\mu}_{<\lambda}(\B)$ is determined for $\mu < \lambda$.\end{theorem}

\begin{proof} We will show how to set up the argument so that it reduces to the proof of \autoref{zapletal-answer}.

Let $\seq{\kappa_\alpha}{\alpha \in \ON}$ be the sequence inductively defined so that $\kappa_0$ is the least supercompact, $\kappa_\alpha = \sup_{\beta<\alpha}\kappa_\beta$ for limits $\alpha$, and $\kappa_{\alpha+1}$ is the least supercompact above $\kappa_\alpha$. Let $\P = \seq{\P_\alpha}{\alpha \in \ON}$ be the Easton-support class iteration such that $\P_0 = \Col(\omega_1,<\!\kappa_0)$, and if $\alpha=\beta+1$, then $\P_\alpha = \P_\beta \ast \dot{\Col}(\aleph_{\alpha+1},<\!\kappa_*)$ where $\kappa_*$ is the least supercompact above $\aleph_{\alpha+1}$ in the extension by $\P_\alpha$.

Let $G$ be $\P$-generic over $V$. Observe that $\GCH$ holds in $V[G]$. Let $V[G_\alpha]$ denote the induced $\P_\alpha$-generic submodel of $V[G]$. We will argue that $V[G]$ witnesses the statement of the theorem.

We argue by induction on $\alpha$ that $\kappa_\alpha = \aleph_{\alpha+2}$ in $V[G_\alpha]$. Since $\P/G_\alpha$ is $\aleph_{\alpha+2}$-closed, this implies that $\kappa_\alpha = \aleph_{\alpha+2}$ in $V[G]$. The case $\alpha=0$ is immediate.  If $\alpha = \beta+1$, then it is the case that in $V[G_\beta]$, the induced $\P_\beta$-generic submodel, it holds that $\kappa_\beta = \aleph_{\beta+2}=\aleph_{\alpha+1}$. So the statement follows because $V[G_\alpha] = V[G_\beta][K]$ where $K$ is $\Col(\aleph_{\alpha+1},<\kappa_\alpha)$-generic over $V[G_\beta]$. Then the case for limit $\alpha$ follows easily by induction, where we get $V[G_\alpha] \models \bq \kappa_\alpha = \sup_{\beta<\alpha}\kappa_\beta = \sup_{\beta<\alpha}\aleph_\beta = \aleph_\alpha\eq$.

Now let $\lambda$ be a cardinal in $V[G]$ and let $\alpha$ be such that $\aleph_{\alpha+1} = \lambda$ in $V[G]$. Let $\B$ be a complete $(<\!\lambda,\mu)$-distributive Boolean algebra of cardinality $\aleph_\Theta$. So $\lambda \le \aleph_\Theta$, and by $\GCH$ the set of possible plays in $\G_{<\lambda}^{\mu}(\B)$ has cardinality $\aleph_\Theta$. Hence we only need to argue in $V[G_\Theta]$, which contains $\B$ by closure. Write $\P_\Theta = \P_\textup{low} \ast \Col(\check{\aleph}_{\alpha+1},<\kappa_\alpha) \ast \dot{\P}^\textup{high}$, and write $V[G] = V[G_\textup{low}][H][G^\textup{high}]$ where $H$ is $\Col(\aleph_{\alpha+1},<\!\kappa_\alpha)$-generic over $V[G_\textup{low}]$.

In $V'=V[G_\textup{low}]$, $\kappa_\alpha$ retains its supercompactness by smallness of $\P_\textup{low}$. It is sufficient to show that $\Col(\aleph_{\alpha+1},<\kappa_\alpha) \ast \dot{\P}^\textup{high}$ forces that $\gchoose$ has a winning strategy in $\G_{<\lambda}^{\mu,<\nu}(\B)$. Take a $\kappa_\Theta$-supercompact embedding $j:V' \to M$ with critical point $\kappa_\alpha$. Then the quotient of $\P_\Theta$ with $G_\textup{low}$ takes the form $\Col(\aleph_{\alpha+1},<\kappa_\alpha) \ast \dot{\P}^\textup{high}$ where $\dot{\P}^\textup{high}$ is $\aleph_{\alpha+1}$-closed. Write $j(\Col(\aleph_{\alpha+1}),<\!\kappa_\alpha) = \Col(\aleph_{\alpha+1},<\!\kappa_\alpha) \times \T$. 

We recall the absorption theorem:

\begin{fact}[Absorption Theorem] Suppose $\kappa$ is a regular cardinal and that $\P$ is a separative and $\kappa$-strongly strategically closed poset such that $|\P|<\lambda$. Then there is a complete embedding $\iota: \P \to \Col(\kappa,<\lambda)$ such that if $G$ is $\P$-generic over $V$, then $\Col(\kappa,<\lambda)$ is forcing-equivalent to $\Col(\kappa,<\lambda)/\iota(G)$. Moreover, this works if $\Col(\kappa,<\lambda)$ is replaced by $\Col(\kappa,A)$ where $\sup A = \lambda$ (see \cite[Section 14]{Handbook-Cummings}).\end{fact}

Applying absorption, there is a filter $K$ which is $\T$-generic such that
\[
V[G_\textup{low}][H][G^\textup{high}][K] = V[G_{\textup{low}}][H][K],
\]
and we therefore obtain the lift
\[
j:V[G]=V[G_\textup{low}][H][G^\textup{high}] \to V[G_\textup{low}][H][K].
\]
Since $K$ is adjoined with a $\aleph_{\alpha+1}$-closed forcing, \autoref{closure-preservation1} implies that is the enough to obtain a winning strategy for $\gcut$ in $V[G_\textup{low}][H][K]$. From here it is possible to use the lifted $j$ and the lifting argument from the rest of the proof of \autoref{zapletal-answer}.\end{proof}

\section{Undetermined Cut and Choose Games}\label{sec-andmore}

This section will use various combinatorial principles, so we refer the reader to some in-depth surveys \cite{AST-Magidor,Handbook-Eisworth,Cummings2005}.

\subsection{Approachability and the Cut and Choose Game}

Now we will sharpen some results on the existence of undetermined cut and choose games in successor form, meaning that they have a last round of the game.

We recall the definition of the approachability property.

\begin{definition}\label{approach-def} Let $\kappa$ be regular and let $\vec{a}=\seq{a_\alpha}{\alpha<\kappa}$ be a sequence of bounded subsets of $\kappa$. An ordinal $\gamma<\kappa$ is \emph{approachable} with respect to $\vec{a}$ if there is an unbounded subset $A \subset \gamma$ of order type $\cf( \gamma)$ such that for all $\beta<\gamma$ there is $\alpha<\gamma$ with $A\cap\beta=a_\alpha$. The \emph{approachability property} holds at $\kappa$ if there is a club $C \subseteq \kappa$ such that every $\gamma\in C$ is approachable with respect to $\vec{a}$, and we denote this $\kappa \in I[\kappa]$.\end{definition}

It is well-known that if $\kappa=\lambda^+$ and $\lambda^{<\lambda} = \lambda$, then $\kappa \in I[\kappa]$.

\begin{theorem}\label{ap-undetermined} Suppose that $\lambda$ is regular and that $\kappa=\lambda^+$. If $\kappa \in I[\kappa \cap \cof(\lambda)]$, then there is a complete $\lambda$-closed Boolean algebra $\B$ such that:

\begin{itemize}[$\circ$]
\item for all $\mu,\nu$ such that $\nu \le \lambda^+ \le \mu$, the game $\G^{\mu,<\nu}_\lambda(\B)$ is undetermined, and 
\item for all $\mu \ge 2$, $\G^{\mu}_\lambda(\B)$ is undetermined.
\end{itemize}\end{theorem}

This follows a series of results of Dobrinen (also in joint work with Cummings) \cite{Dobrinen2003, Dobrinen2007,Dobrinen2008,Cummings-Dobrinen2007}. Some of these theorems pertain to obtaining Suslin algebras, which can fail to exist in the presence of approachability. \autoref{ap-undetermined} extends the last theorem in this series (\cite[Theorem 23]{Dobrinen2008}) by replacing the hypothesis $\lambda^{<\lambda}=\lambda$ with $\kappa \in I[\kappa \cap \cof(\lambda)]$.

To prove this theorem, we first provide a new version of the lemma in the source material, which we will use to address the first collection of cases.

\begin{lemma}\label{first-approach} Suppose that $\lambda$ is regular and $\lambda^+ \cap \cof(\lambda) \in I[\lambda^+]$. Suppose also that $\B$ is a complete Boolean algebra such that $\gchoose$ has a winning strategy for $\G^{\mu,<\nu}_\lambda(\B)$ where $\nu \le \lambda^+ \le \mu$.  Then $\B$ preserves stationary subsets of $\lambda^+ \cap \cof(\lambda)$.\end{lemma}

\begin{proof}

Let $\dot{C}$ be a $\B$-name for a club in $\kappa$ as forced by some $\bar{b} \in \B$. We will argue that there is some $b \in \B$ such that $b \Vdash \bq \delta \in \dot{C}\eq$.

Let $\kappa = \lambda^+$ and fix a stationary set $S \subseteq \kappa \cap \cof(\lambda)$. Let $\vec{a}=\seq{a_\alpha}{\alpha<\kappa}$ and $D \subseteq \kappa$ witness approachability. Choose $\delta \in S$ so that $\delta$ is approachable with respect to $\vec{a}$ and so that $\delta = M \cap \kappa$ for an elementary submodel $M \prec H(\Theta)$ where $\Theta$ is sufficiently large (which can be done because the set of such $\delta$ is a club). Assume that $\dot{C},\bar{b} \in M$ and that $\sigma \in M$ where $\sigma$ is a winning strategy for $\gchoose$.

Let $A \subseteq \delta$ be unbounded with order-type equal to $\lambda$ such that for all $\beta<\delta$, there is $\gamma<\delta$ such that $A \cap \beta = a_\gamma$. Let $\seq{\beta_\xi}{\xi<\lambda}$ enumerate $A$.

We will construct a run
\[
\vec{s} = \bar{b} {}^\frown \seq{\{b_i^\xi : i< \kappa \},B_\xi}{\xi < \lambda}
\]
in $\G^{\kappa,<\nu}_\lambda(\B)$ (without loss of generality, we can let $\mu=\kappa$) by induction on $\xi<\lambda$ in such a way that
\[
\seq{\{b_i^\xi : i< \kappa \},B_\xi}{\xi < \eta} \in M
\]
for all $\eta<\lambda$. Specifically, we will define $\gcut$'s moves and let $\gchoose$'s moves be dictated by the strategy $\sigma$. For some $\xi<\lambda$ and let $\seq{\gamma_i}{i<\kappa}$ enumerate the set of $\gamma$ such that there is
\[
c \le \bar{b} \wedge \left(\bigwedge_{\zeta<\xi}\bigvee_{i \in B_\zeta}b_i^\zeta \right)
\]
forcing $\bq \min(\dot{C} \setminus (\beta_\xi+1)) = \gamma \eq$. Then let
\[
b^\xi_i = \bar{b} \wedge \left(\bigwedge_{\zeta<\xi}\bigvee_{i \in B_\zeta}b_i^\zeta \right) \wedge  \bigvee\{c \in \B \mid c \le \textup{ and } c \Vdash \bq \min(\dot{C} \setminus (\beta_\xi+1)) = \gamma_i \}
\]
and let $\gcut$ play the partition $\{b_i^\xi : i < \kappa\}$ on the $\xi$'th round. Since this move is defined from $\sigma$ and $\seq{\beta_\zeta}{\zeta \le \xi}$, it is an element of $M$.

Now let
\[
b_* = \bar{b} \wedge \left(\bigwedge_{\xi<\lambda}\bigvee_{i \in B_\xi}b_i^\xi \right).
\]
We argue that $b_* \Vdash \bq \delta \in \dot{C} \eq$: For each $\xi$, $|B_\xi|<\kappa$, it follows that $\sup B_\xi < \kappa$, which together with the fact that $B_\xi \in M$ implies that $\sup B_\xi < \delta$. This implies that $b_* \Vdash \bq \dot{C} \cap (\beta_\xi,\delta) \ne \emptyset \eq$ for all $\xi<\lambda$, which implies our claim by closure.\end{proof}

Now we prepare to address the second set of cases. We need the following from work of Dobrinen:

\begin{fact}\label{big-zapletal} Suppose that $|\mathcal{P}(2^{<\lambda})|=2^\lambda$ and that $\mathbb{A}$ is a $<\!\lambda^+$-complete Boolean algebra of cardinality $\lambda^+$. Then if $\gchoose$ has a winning strategy in $\G_\lambda^\candc(\mathbb{A})$, it follows that $\gchoose$ has a winning strategy in $\G_\lambda^{2^\lambda}(\B)$  \cite[Theorem 16]{Dobrinen2007}.\end{fact}

In the original, this theorem is stated in a context where $\B$ is assumed to be complete. However, examination of the proof shows that $<\lambda^+$-completeness is sufficient.

We will also use Dobrinen's direct generalization of Veli{\v c}kovi{\'c} \cite[Theorem 2.1]{Velickovic1986}:

\begin{fact}\label{cut-to-bm} If $\gchoose$ has a winning strategy in $\G_\lambda^{\infty}(\B)$, then $\B$ is $(\lambda+1)$-strategically closed  \cite[Theorem 29]{Dobrinen2007}.\end{fact}

Putting this information together, we get:

\begin{proposition} Suppose that $|\mathcal{P}(2^{<\lambda})|=2^\lambda$ and that $\mathbb{A}$ is a $<\!\lambda^+$-complete Boolean algebra of cardinality $\lambda^+$. Then if $\gchoose$ has a winning strategy in $\G_\lambda^\candc(\mathbb{A})$, it follows that $\mathbb{A}$ is $(\lambda+1)$-strategically closed.\end{proposition}

\begin{proof} Use \autoref{big-zapletal}, then observe that since $\lambda^+$ is the maximum possible size of a partition, this means that $\gchoose$ has a winning strategy in the variation of the game $\G^\infty_\lambda(\B)$, then apply \autoref{cut-to-bm}.\end{proof}

We will also need:

\begin{proposition} If $\lambda$ is regular and $\P$ is a $<\!\lambda$-closed forcing, $\P$ forces that a ground model poset $\mathbb{A}$ is $\gamma$-strategically closed for some $\gamma<\lambda$, then in the ground model $\mathbb{A}$ is $\gamma$-strategically closed.\end{proposition}

We sketch an argument along the lines of \autoref{closure-preservation1}.

\begin{proof} Let $\bar{p}$ force that $\dot{\sigma}$ is a name for a winning strategy for player $\textsf{II}$. For each $A \in \mathbb{A}$, take some large enough $\Theta$ and let $\{b_\alpha^a : \alpha < \Theta \}$ enumerate $\mathbb{A} \rest a$, possibly with repetitions.

Build an assignment on ${}^{<\gamma}\Theta$ of the form $t \mapsto (a_t,p_t) \in \mathbb{A} \times \P$ with the property that $a_{t {}^\frown \langle \alpha \rangle} \in \{b_\alpha^{a_t} : \alpha < \Theta\}$. Each $t \in {}^{<\gamma}\Theta$ will be associated to a play of the game given by $\vec{s}_t:=\seq{a_{t \rest \xi},a^{t \rest \xi}_{t(\xi)}}{\xi \in \dom(t)}$.  We construct the assignment so that \emph{(1)} if $u \sqsupseteq t$ then $p_u \le_\P p_t$ and \emph{(2)} if $\dom(t)=\xi+1$ and $u {}^\frown \langle \alpha \rangle = t$, then $p_t \Vdash \bq \dot{\sigma}(\vec{s}_u {}^\frown b_\alpha^{a_u}) = a_t$.

The constructed assignment will be equivalent to a winning strategy in the ground model.\end{proof}

Lastly, we make an observation in the form of an easy proposition:

\begin{proposition}\label{easy-cnc} Let $\B$ be a complete Boolean algebra, let $\mu,\nu$ be cardinals, and let $\gamma<\delta$.

\begin{itemize}[$\circ$]
\item Suppose $\mu \le \mu'$ and $\nu \ge \nu'$. If $\gcut$ has a winning strategy for $\G^{\mu,<\nu}_{<\gamma}(\B)$, then $\gcut$ has a winning strategy for $\G^{\mu',<\nu'}_{<\delta}(\B)$.
\item Supose $\mu \ge \mu'$ and $\nu \le \nu'$. If $\gchoose$ has a winning strategy for $\G^{\mu,<\nu}_{<\delta}(\B)$, then $\gchoose$ has a winning strategy for $\G^{\mu',<\nu'}_{<\gamma}(\B)$.

\end{itemize}
\end{proposition}

Now we can prove the theorem:

\begin{proof}[Proof of \autoref{ap-undetermined}] Let $S \subseteq \kappa \cap \cof(\lambda)$ be stationary and costationary. Let $\P$ be the poset of closed bounded subsets $c$ of $\kappa$ such that $c \cap S = \emptyset$, and let $\B$ be the regular open algebra of $\P$. We will argue that $\G_\lambda^{\candc}(\B)$ is undetermined.

First, since $\kappa \in I[\kappa]$, $\B$ is $<\!\kappa$-distributive by an argument that is strictly easier than the one in \autoref{undetermined-form} below, so $\gcut$ does not have a winning strategy in $\G_\lambda^\infty(\B)$ (this also appears in \cite[page 718]{Foreman1983}). It follows from \autoref{easy-cnc} that $\gcut$ does not have a winning strategy in $\G^{\mu,<\nu}_\lambda$ for any $\mu,\nu$, so this takes care of both collections of cases.

Suppose for contradiction that $\gchoose$ has a winning strategy in $\G^{\mu,<\nu}_\lambda$ where $\nu \le \lambda^+ \le \mu$. Then \autoref{first-approach} implies that $\B$ preserves stationarity of $\S$, which is clearly a contradiction since $\P$ explicitly adds a club through the complement of $S$.

Suppose for contradiction that $\gchoose$ has a winning strategy for $\G_\lambda^\mu(\B)$ for $\mu \le \lambda$ ($\mu > \lambda$ is covered by the other cases). Then $\gchoose$ has a winning strategy for $\G_\lambda^\candc(\B)$ (again \autoref{easy-cnc}). We claim that $\mathbb{B}$ is $(\lambda+1)$-strategically closed. It is enough to prove that this assertion holds in an extension by $\Col(\lambda^+,|2^\lambda|)$: The collapse preserves $<\!\lambda^+$-completeness of $\mathbb{B}$ and forces $|\mathcal{P}(2^{<\lambda})| = 2^\lambda$, so we can apply \autoref{cut-to-bm} in the extension to conclude that $\mathbb{B}$ is $(\lambda+1)$-strategically closed in the extension. Using $\kappa \in I[\kappa]$, it follows that $\B$ preserves stationarity of all stationary subsets of $\kappa \cap \cof(\lambda)$, which implies that $\B$ preserves stationarity of $S$. But this is again a clear contradiction as in the previous set of cases.\end{proof}

\subsection{More Models with Undetermined Games}

\begin{theorem}\label{lower-bound} If $\lambda$ is regular and $\kappa=\lambda^+$ is not weakly compact in $L$, then there is a complete Boolean algebra $\B$ such that $\G^{\candc}_\lambda(\B)$ is not determined.\end{theorem}

Recall that a stationary subset $S \subseteq \kappa$ \emph{reflects} at $\delta \in \kappa \cap \cof(>\omega)$ if $S \cap \delta$ is stationary in $\delta$.
%The main observation for this section is the following:

\begin{proposition}\label{undetermined-form} Let $\lambda$ be regular and let $\kappa = \lambda^+$. Suppose that there are stationary sets $S \subseteq \kappa^+ \cap \cof(\omega)$ and $T \subseteq \kappa \cap \cof(\lambda)$ such that $T \in I[\kappa]$ and $S$ does not reflect at any point in $T$. Then there is a Boolean algebra $\B$ such that $\G^{\candc}_\lambda(\B)$ is undetermined.\end{proposition}

This proof will use a theorem of Zapletal.

\begin{fact}[Zapletal]\label{zapletal-thm} If $\B$ is a \emph{countably complete} Boolean algebra and $\gchoose$ has a winning strategy for $\G^{\candc}_\omega(\B)$, then $\B$ is semiproper \cite[Theorem 1]{Zapletal1995}.\end{fact}

As in \autoref{big-zapletal}, the theorem is stated for complete Boolean algebras, but  that examination of the proof shows that the theorem holds for countably complete Boolean algebras because only countable meets and joins are used.

\begin{proof}[Proof of \autoref{undetermined-form}] Given $S$ and $T$ as stated, let $\P$ be the poset of closed bounded subsets $c$ of $\kappa$ such that $c \cap S = \emptyset$, and let $\B:=\ro(\P)$ be its regular open algebra. We claim that $\B$ witnesses the proposition.

We argue that $\gchoose$ does not have a winning strategy in the game $\G_\lambda^{\candc}(\B)$. If there were such a strategy, there would be a winning strategy for $\gchoose$ in $\G_\omega^{\candc}(\B)$. Let $G$ be $\Col(\omega_1,\kappa)$-generic and work in $V[G]$: Then $S$ is still stationary in $\kappa$ (see \cite[Proposition 2.16]{Handbook-Eisworth}) and $\gchoose$ still has a winning strategy in $\G_\omega^{\candc}(\B)$ since no new plays of the game have been added. Let $\Phi:\omega_1 \to \kappa$ be strictly increasing, continuous, and cofinal in $\kappa$, and let $\bar{S}:= \{\alpha < \omega_1 \mid \Phi(\alpha) \in S\}$. Then forcing with $\P$, hence 
$\B$ (the density relation still holds), destroys stationarity of $\bar{S}$. Therefore, $\B$ is not semiproper in $V[G]$. Since $\B$ remains countably complete in $V[G]$, this is a contradiction of \autoref{zapletal-thm}.

It is also the case that $\gcut$ does not have a winning strategy in $\G_\lambda^{\candc}(\B)$ because $\B$ is $<\!\kappa$-distributive (\autoref{foreman-generalization}). The distributivity itself follows by a standard argument: Fix a sequence $\vec{a} = \seq{a_\alpha}{\alpha<\kappa}$ of $[\kappa]^{<\kappa}$ and suppose $\dot{f}$ is a $\P$-name for a function $\kappa \to \ON$. Choosing $\Theta$ large enough for $H(\Theta)$ to contain the sets in the following discussion, take an elementary submodel $M \prec H(\Theta)$ with $\vec{a} \in M$ and $M \cap \kappa = \delta \in T$ approachable with respect to $\vec{a}$. Let $D \subseteq \delta$ witness that $S$ is nonreflecting in $\delta$. Let $A \subseteq \delta$ cofinal with $\ot(A) = \cf(\delta)$ witness approachability. Construct a $\le_\P$-decreasing sequence $\seq{p_\xi}{\xi<\cf(\delta)} \subset \P$ such that $p_\xi$ decides a value for $\dot{f}(\xi)$, and use that fact that initial segments of $A$ are in $M$ and the fact that $D$ is closed in order to continue the construction at limit stages. Then $\bar{p} := \{\delta\} \cup \bigcup_{\xi<\cf(\delta)}p_\xi$ is an element of $\P$ and forces that $\dot{f}=g$ where $g(\xi)=\beta_\xi$ if and only if $p_\xi \Vdash \bq\dot{f}(\xi)=\beta_\xi \eq$.\end{proof}

Towards the proof of \autoref{lower-bound}, we provide a proof of a folklore result.

\begin{proposition}\label{easy-approach} If $\kappa^{<\kappa}=\kappa$, then for all regular $\tau<\kappa$, there is a stationary set $S \subseteq \kappa \cap \cof(\tau)$ which is approachable.\end{proposition}

\begin{proof} Fix an enumeration $\vec{a} = \seq{a_\alpha}{\alpha<\kappa}$ of $\kappa^{<\kappa}$. Let $C \subseteq \kappa$ be a club. Inductively define a continuous sequence $\seq{\delta_\xi}{\xi<\tau} \subseteq C$ so that if $\seq{\delta_\xi}{\xi \le \zeta}$ has been defined and equal to $a_\alpha$, then $\delta_{\zeta+1}$ is chosen to be larger than $a_\alpha$. Then if $\delta = \sup_{\xi<\tau}$, then clearly $\delta \in C$ and $\seq{\delta_\xi}{\xi<\tau}$ witnesses approachability of $\delta$.\end{proof}

\begin{proof}[Proof of \autoref{lower-bound}] We  will list some facts that we will apply, then identify a candidate for $\B$, and then indicate why it witnesses the theorem.

For these facts, assume $\kappa$ is a regular cardinal.

\begin{enumerate}

\item  If $\kappa$ is not weakly compact in G{\"o}del's Constructible Universe $L$, then the principle $\square(\kappa)$ holds \cite[1.10]{Todorcevic1987}.

\item If $\square(\kappa)$ holds and $\P$ is a countably closed forcing, then $\square(\kappa)$ holds after forcing with $\P$ (this is a variant of, e.g.\ \cite[Lemma 4.5]{AST-Magidor}).

\item If $\square(\kappa)$ holds and $\tau < \kappa$ is regular and $S \subseteq \kappa \cap \cof(\tau)$ is stationary, then there are stationary sets $S_0,S_1 \subseteq S$ that do not reflect simultaneously (many variants exist, e.g.\ in \cite[Theorem 2.2]{Hayut-LambieHanson2017}).

\end{enumerate}

If $\kappa=\lambda^+$ is not weakly compact in $L$, then we have $\square(\kappa)$ by (1). Then, forcing with $\Col(\kappa,\kappa^{<\kappa})$ if necessary, we preserve $\square(\kappa)$ by (2) and guarantee that there is an approachable stationary set $T \subseteq \kappa \cap \cof(\lambda)$ by \autoref{easy-approach}.

Using (3), let $S_0,S_1 \subseteq \kappa \cap \cof(\omega)$ be stationary subsets that do not reflect simultaneously. Then for some $\epsilon \in \{0,1\}$ and some stationary $T' \subseteq T$, $S_\epsilon$ does not reflect at points in $T'$: otherwise there would be clubs $C_0,C_1 \subseteq \kappa$ such that $S_\epsilon$ reflects at all points in $T' \cap C_\epsilon$ (for $\epsilon \in \{0,1\}$), and then there would be a point in $T \cap C_0 \cap C_1$ at which $S_0$ and $S_1$ reflect simultaneously, a contradiction.

Hence there are stationary $S \subseteq \kappa \cap \cof(\omega)$ and $T \subseteq \kappa \cap \cof(\lambda)$ such that $T \in I[\kappa]$ and $S$ is nonreflecting in $T$, and so we can apply \autoref{undetermined-form}.\end{proof}

Note that successive failures of $\square(\kappa)$ also imply the consistency of very substantial large cardinals \cite{Schimmerling1995}.

We can show that Martin's Maximum ($\MM$) \cite{Foreman-Magidor-Shelah1989} is compatible with the presence of a Boolean algebra $\B$ for which the cut and choose game of length $\omega_1$ is undetermined. Notably, the proper forcing axiom ($\PFA$), which is weaker than $\MM$, implies that failure of approachability at $\aleph_2$ \cite{Koenig-Yoshinobu2004}.

\begin{theorem} The axiom $\MM^{++}$ is compatible with the existence of a complete Boolean algebra $\B$ such that $\G_{\omega_1}^{\candc}(\B)$ is undetermined.\end{theorem}

\begin{proof} By a theorem of Larson, $\MM$ is preserved under $\omega_2$-directed closed forcing extensions \cite{Larson2000}. Cox observed that this applies to $\MM^{++}$ as well \cite[Theorem 4.7]{Cox2021}.

Therefore, by \autoref{undetermined-form}, it is enough to show that there is a $\omega_2$-directed closed poset adding stationary sets $S \subseteq \omega_2 \cap \cof(\omega)$ and $T \subseteq \omega_2 \cap \cof(\omega_1)$ such that $T$ consists of approachable points and $S$ is nonreflecting in $T$. (The poset we provide will just be equivalent to $\Add(\omega_2)$.)

Fix an enumeration $\seq{a_\alpha}{\alpha<\omega_2}$ of $\omega_2^{<\omega_2}$. Let $\P$ consist of partial functions $p:\omega_2 \rightharpoonup 3$ such that:

\begin{enumerate}
\item $\dom(p) \in \omega_2$ is a successor ordinal,
\item if $p(\alpha)=1$ then $\cf(\alpha)=\omega$,
\item if $p(\alpha)=2$ then $\cf(\alpha)=\omega_1$, $\alpha$ is approachable with respect to $\vec{a}$, and $\{\beta<\alpha \mid p(\beta) = 1\}$ is nonstationary in $\alpha$.
\end{enumerate}

If $G$ is $\P$-generic, we let $S = \{\alpha<\omega_2 \mid \exists p \in G, p(\alpha) = 1\}$ and $T = \{\alpha < \omega_2 \mid \exists p \in G, p(\alpha)=2\}$. The poset $\P$ is $<\!\omega_2$-closed because, given a $\le_\P$-decreasing $\seq{p_\xi}{\xi<\tau}$ with $\delta_\xi = \dom(p_\xi)$, one can obtain a lower bound with $p(\delta)=0$ where $\delta = \sup_{\xi<\tau}\delta_\xi$.

We will prove the stationarity of $T$ since it is slightly more involved than stationarity of $S$, which follows from relatively standard arguments. Take some $p_0 \Vdash \bq \dot{C}$ a club subset of $\omega_2\eq$, then build a $\le_\P$ decreasing sequence $\seq{p_i}{i<\omega_1}$ and a sequence of ordinals $\seq{\alpha_i}{i<\omega_1}$ interleaved with $\seq{\max \dom p_i}{i<\omega_1}$ such that $p_i \Vdash \bq \alpha_i \in \dot{C}\eq$, and moreover at successor steps make sure that $\alpha_{i+1}$ is large enough so that if $\seq{\alpha_j}{j<i}=a_\xi$, then $\xi<\alpha_{i+1}$. Then choose a lower bound $p_*$ such that, with $\alpha_* = \sup_{i<\omega_1}\alpha_i$, we have $p_* \Vdash \bq \alpha_* \in \dot{T} \eq$, and where we necessarily have $p_* \Vdash \bq \alpha_* \in \dot{C} \eq$.\end{proof}

We can also look for more undetermined games. \autoref{global-zapletal}, for example, leaves a question about successors of singular cardinals.

\begin{conjecture} The consistency of $\ZFC$ implies the consistency of the statement that there is some complete $\B$ such the game $\G^\candc_{<\aleph_{\omega+1}}(\B)$ is undetermined.\end{conjecture}

\subsection*{Acknowledgments} Many thanks to Heike Mildenberger for reading early versions of \autoref{zapletal-answer}. I would also like to thank Tom Benhamou for enriching conversations that overlapped with the notions in this paper.

\bibliographystyle{alpha}
\bibliography{bibliography}

\end{document}